\documentclass[12pt]{article}
\usepackage[centertags]{amsmath}
\usepackage{amsfonts}
\usepackage{amssymb}
\usepackage{amsthm}
\usepackage{amsmath,mathrsfs}
\usepackage{amsmath,amscd}
\usepackage{mathtools}
\usepackage[matrix,arrow,curve,cmtip]{xy}
\title{\textbf{Modular Model Categories}}
\author{Renaud Gauthier \footnote{rg.mathematics@gmail.com} \\ \\}
\theoremstyle{definition}
\newtheorem{ShSqu}{Proposition}[subsection]
\newtheorem{dir2}{Lemma}[subsection]
\newtheorem{et}{Lemma}[subsection]
\newtheorem*{acknowledgments}{Acknowledgments}
\newtheorem{NisCov}{Lemma}[subsection]

\newcommand{\beq}{\begin{equation}}
\newcommand{\eeq}{\end{equation}}
\newcommand{\rarr}{\rightarrow}

\newcommand{\larr}{\leftarrow}

\newcommand{\eset}{\emptyset}

\newcommand{\xrarr}{\xrightarrow}

\newcommand{\cC}{\mathcal{C}}
\newcommand{\cCop}{\cC^{\op}}
\newcommand{\cD}{\mathcal{D}}

\newcommand{\cF}{\mathcal{F}}
\newcommand{\cG}{\mathcal{G}}

\newcommand{\cI}{\mathcal{I}}

\newcommand{\cM}{\mathcal{M}}

\newcommand{\cO}{\mathcal{O}}
\newcommand{\cP}{\mathcal{P}}

\newcommand{\cU}{\mathcal{U}}

\newcommand{\cX}{\mathcal{X}}
\newcommand{\cY}{\mathcal{Y}}
\newcommand{\cZ}{\mathcal{Z}}

\newcommand{\gm}{\mathfrak{m}}

\newcommand{\bA}{\mathbb{A}}
\newcommand{\bE}{\mathbb{E}}

\newcommand{\bZ}{\mathbb{Z}}
\newcommand{\Cat}{\text{Cat}}

\newcommand{\deem}{\text{dim}}
\newcommand{\Fun}{\text{Fun}}
\newcommand{\Hom}{\text{Hom}}

\newcommand{\map}{\text{map}}

\newcommand{\op}{\text{op}}

\newcommand{\Set}{\text{Set}}
\newcommand{\SetD}{\Set_{\Delta}}
\newcommand{\SetDstar}{\Set_{\Delta, *}}
\newcommand{\Spec}{\text{Spec}}

\newcommand{\PreSh}{\text{PreSh}}
\newcommand{\Sh}{\text{Sh}}

\newcommand{\bExt}{\bE \text{xt}}
\newcommand{\bExtdot}{\bExt^{\ctrdot}}
\newcommand{\Bdot}{B_{\ctrdot}}
\newcommand{\ctrdot}{\centerdot}
\newcommand{\Catfes}{\Cat^{full}_{ess.surj.}}
\newcommand{\Cdot}{C_{\ctrdot}}
\newcommand{\cOX}{\cO_X}
\newcommand{\cOY}{\cO_Y}
\newcommand{\cOXi}{\cO_{X_i}}

\newcommand{\cFdot}{\cF_{\ctrdot}}
\newcommand{\cIji}{\cI_{ji}}
\newcommand{\cIdot}{\cI^{\ctrdot}}
\newcommand{\cCdot}{\cC_{\ctrdot}}
\newcommand{\Dji}{\Delta_{ji}}

\newcommand{\Ext}{\text{Ext}}
\newcommand{\Extdot}{\Ext^{\ctrdot}}
\newcommand{\Funfes}{\Fun^{full}_{ess.surj.}}
\newcommand{\fa}{f_{\alpha}}
\newcommand{\faji}{f_{\alpha ji}}
\newcommand{\gammalk}{\gamma_{lk}}

\newcommand{\xH}{\text{H}}
\newcommand{\HH}{\text{HH}}
\newcommand{\HHdot}{\HH^{\ctrdot}}

\newcommand{\Hdot}{\xH^{\ctrdot}}

\newcommand{\Iaj}{I_{\alpha j}}

\newcommand{\Ia}{I_{\alpha}}

\newcommand{\Jai}{J_{\alpha i}}

\newcommand{\muaij}{\mu_{\alpha i_j}}

\newcommand{\muai}{\mu_{\alpha i}}
\newcommand{\muaj}{\mu_{\alpha j}}
\newcommand{\mijk}{m_{i_{j_k}}}
\newcommand{\mij}{m_{i_j}}
\newcommand{\njk}{n_{j_k}}
\newcommand{\Nis}{\text{Nis}}
\newcommand{\OXX}{\cO_{X \times X}}
\newcommand{\paij}{p_{\alpha ij}}
\newcommand{\pai}{p_{\alpha i}}

\newcommand{\phiji}{\phi_{ji}}
\newcommand{\psikj}{\psi_{kj}}

\newcommand{\raji}{r_{\alpha ji}}
\newcommand{\saji}{\sigma_{\alpha ji}}

\newcommand{\Smk}{\text{Sm}/k}
\newcommand{\SmCor}{\text{SmCor}}

\newcommand{\sumI}{\sum_{i \in I}}
\newcommand{\sumJ}{\sum_{j \in J}}

\newcommand{\uFi}{\underline{\cF}_i}

\newcommand{\uU}{\underline{U}}
\newcommand{\uF}{\underline{\cF}}
\newcommand{\uX}{\underline{X}}
\newcommand{\Uai}{U_{\alpha i}}
\newcommand{\uai}{u_{\alpha i}}
\newcommand{\Uaj}{U_{\alpha j}}
\newcommand{\ua}{u_{\alpha}}
\newcommand{\Ua}{U_{\alpha}}
\newcommand{\Ub}{U_{\beta}}

\newcommand{\ZSmk}{\mathbb{Z}\text{Sm}/k}

\begin{document}
\maketitle
\begin{abstract}

	To any model category $\cM$, we associate a modular model category, a functor of point $\cM[-]: \Cat \rarr \Cat$, that associates to any small category $\cC$ a functor category $\cM[\cC] = \Fun^{full}_{ess. surj.}(\cC, \cM)$, providing parametrizations of a same model category $\cM$ by different small categories. We are in particular interested in using schemes as parameters. We consider $\cC = \ZSmk$ the category of linear combinations of smooth separated schemes of finite type over $\Spec(k)$, $k$ a field, referred to as $\bZ$-schemes. We contrast this with using the $\bA^1$-homotopy category of $\bZ$-schemes as a parametrizing category.
\end{abstract}

\newpage

\section{Introduction}
Defining a model structure on a category $\cM$ necessitates the introduction of a class of weak equivalences, among other things. From there one may consider different types of weak equivalence between fixed objects, each one being defined relative to some algebraic invariant. The idea here is to introduce some variability in the notion of weak equivalence we use. If the model structure on $\cM$ is fixed however, so are the equivalences. The variety in weak equivalences can nevertheless be implemented with parametrizations. From there, one is led to considering parametrizations of $\cM$ by small categories $\cC$, provided by $\Fun^{full}_{ess.surj.}(\cC, \cM) \in \cP(\cM)$, $\cP(\cM)$ the category of parametrizations of $\cM$. If $\Catfes$ is the category of small categories and full, essentially surjective functors, we have a functor of points $\cM[-]: \Catfes \rarr \cP(\cM)$, that associates to any small category $\cC$ the functor category $\Hom_{\Catfes}(\cC, \cM) = \Funfes(\cC, \cM) = \cM[\cC]$, which we refer to as the modular model category associated to $\cM$. Note that this inscribes itself within the theory of parametrized homotopy theory (\cite{MS}). We are in particular interested in using schemes as parameters. Simple categories were already parametrized by schemes in \cite{Ma} for representationt theoretic purposes. Most recently, the topos $\cM^{(\cX,I)}$ was considered in $\cite{Ka}$, where $\cM$ is a left proper combinatorial simplicial model category, $\cX$ a site with interval $I$, with aim the construction of algebraic cobordism for motivic stacks, which is achieved by letting $\chi=\Smk$, $I = \bA^1$ and $\cM = \SetDstar$. In the present work we use linear combinations of schemes as parameters, which we refer to as $\bZ$-schemes. This is implemented by first generalizing schemes to finite correspondences (\cite{FSV}), and then to $\bZ$-schemes. Aside from providing a generalization, note that morphisms between schemes in $\SmCor(k)$ are finite correspondences, linear combinations of schemes. Thus by taking $\bZ$-schemes as objects, we place ourselves at any level in the $\infty$-category $\SmCor(k)$, an obvious generalization being $\ZSmk$, the $\infty$-category of $\bZ$-schemes. Using Yoneda we regard those as presheaves. We then consider full, essentially surjective functors from $\Sh(\ZSmk, \Nis)$ into a given model category $\cM$. This is one object of $\cP(\cM)$, the parametrization of $\cM$ by $\bZ$-schemes. Next, we consider categories $\cC$ endowed with an equivalence relation as an alternative to definning categories with an interval object. Define two morphisms $\phi: X \rarr Y$ and $\phi': X' \rarr Y'$ in $\cC$ to be equivalent if $X \sim X'$ and $Y \sim Y'$ in $\cC$. Let $\cC/\!\!\sim$ be the category of equivalence classes of objects of $\cC$ with equivalence classes of morphisms between them. For $F: \cC \rarr \cM$ a functor, define $FX \sim FX'$ if $X \sim X'$ in $\cC$, and $F\phi \sim F\phi'$ if $\phi \sim \phi'$. Let $\cM /\!\!\sim$ be the category $\cM$ modulo those equivalence relations. We have an induced functor $[F]: \cC/\!\!\sim \,\,\rarr \cM/\!\!\sim$. We apply this formalism to $\cC = \Sh(\ZSmk, \Nis)$ in particular. For that purpose, we consider a notion of equivalence relation on $\ZSmk$, and one in particular we use is based on the Hochschild cohomology of schemes, generalized to $\bZ$-schemes. Probably the easiest way to define it is by $HH^{\ctrdot}(X) = \Ext^{\ctrdot}_{\cO_{X \times X}}(\Delta_* \cOX, \Delta_* \cOX)$, where if $X = \sum m_i [X_i]$ and $Y = \sum n_j[Y_j]$ are elements of $\ZSmk$, we say $X \sim Y$ if and only if the indexing sets are the same, $m_i = n_i$ for all $i$, and $HH^{\ctrdot}(X_i) \cong HH^{\ctrdot}(Y_i)$ for all $i$, resulting in $HH^{\ctrdot}(X) \cong HH^{\ctrdot}(Y)$ since $HH^{\ctrdot}(X) = \otimes m_i HH^{\ctrdot}(X_i)$. Of independent interest, we also define a notion of depth in the topology on $\ZSmk$ for the sake of precision by defining the general notion of powered topology, which goes as follows: suppose we have two categories of objects of a same type,  
$\cX_N = \{ \cX_{N+1} \}$ and $\cY_N = \{ \cY_{N+1} \}$  themselves objects of a category $\cC_{N-1}$ with a functor $F_N: \cY_N \rarr \cX_N$. Suppose we have a Grothendieck topology $\tau_{N-1}$ on $\cC_{N-1}$ with $F_N$ a covering map, element of a covering family in $K(\cX_N)$. Define a loose pre-topology $\tau_N$ on $\cX_N$ by defining loose covering families in $K(\cX_N)$ to be families of morphisms $\cY_{N+1} \rarr \cX_{N+1}$, satisfying the same defining properties as traditional covering families for traditional Grothendieck topologies. We obtain a layered morphism:
\beq
\xymatrix{
	\cY_N \ar[r] &\cX_N \\
	\cY_{N+1} \ar@{.>}[u] \ar[r] & \cX_{N+1} \ar@{.>}[u]
} \nonumber
\eeq
If the top and bottom maps are (loose) covering maps, then such a square would define a covering map in $K(\cX_N, \cX_{N+1})$, thereby defining a notion of powered topology $\tau_N \circ \tau_{N-1}$ on $\cC_{N-1}$. This formalism has the obvious advantage of giving level-wise degrees of precision.\\

In Section 2, we introduce modular model categories. We consider parametrizations of model categories by schemes, so we introduce $\bZ$-schemes in Section 3. In Sections 4 and 5, we discuss sheaves on individual $\bZ$-schemes, and on the site $(\ZSmk, \Nis)$. In Section 6, we introduce layered morphisms and powered topologies, and in Section 7, we add equivalence relations to the picture. In Section 8 we take stock and define $\ZSmk$-parametrizations of model categories, $\cM[\Sh(\ZSmk, \Nis)]$, which we compare with using the $\bA^1$-homotopy theory of $\bZ$-schemes for base category of our parametrizations, thus we contrast using $\cM[\Sh(\ZSmk, [\Nis])]$ with using $\cM[\Sh_{\bA^1}(\ZSmk, \Nis)]$.

\begin{acknowledgments}
The author would like to thank the organizers of the 2019 Exchange of Mathematical Ideas conference, during which part of this work was presented.
\end{acknowledgments}

\section{Modular Model Categories}
Suppose one has a full, essentially surjective functor $F: \cC \rarr \cD$ from one category to another. Consider a morphism $u: a \rarr b$ in $\cD$. Then one can write $u$ as $F(\phi): FX \rarr FY$ for some $X$ and $Y$ in $\cC$ such that $FX = a$ and $FY = b$, with $\phi:X \rarr Y$. Suppose now we have another full, essentially surjective functor $G: \cC \rarr \cD$. Then that same morphism $u$ can be represented as $G(\psi): GW \rarr GZ$, where $W$ and $Z$ are objects of $\cC$ such that $GW = a$ and $GZ = b$, with $\psi: W \rarr Z$ some morphism. Thus we have two different representations of a same morphism of $\cD$ by morphisms of $\cC$. Hence the category $\Funfes(\cC, \cD)$ gives all the parametrizations of $\cD$ relative to $\cC$. Define $\cP(\cD)$ to be the set of elements of the form $\Funfes(\cC, \cD)$ for some small category $\cC$. We apply this formalism to model categories, so we let $\cD = \cM$ be a model category in what follows. Let $\Catfes$ the category of small categories and full, essentially surjective morphisms between them. $G: \cC \rarr \cD$ in $\Catfes$ induces $G^*: \cM^{\cD} \rarr \cM^{\cC}, F \mapsto F \circ G$. This defines a functor between different elements of $\cP(\cM)$ making it into a category, which we call the \textbf{category of parametrizations of $\cM$}, and we have a functor:
\begin{align}
	\cM[-]: &\Catfes \rarr \cP(\cM) \nonumber \\
	&\cC \mapsto \Funfes(\cC, \cM) = \cM[\cC] \nonumber \\
	&(G: \cC \rarr \cD) \mapsto (G^*:\cM[\cD] \rarr \cM[\cC]) \nonumber
\end{align}
which we call a \textbf{modular model category}, namely the one associated with $\cM$.\\

Going back to the general case, consider a functor $F: \cC \rarr \cD$ from one small category to another, and suppose $\cC$ is endowed with an equivalence relation. Define two morphisms $\phi: X \rarr Y$ and $\phi':X' \rarr Y'$ of $\cC$ to be equivalent, $\phi \sim \phi'$, if $X \sim X'$ and $Y \sim Y'$. Define $FX \sim FX'$ in $\cD$, if $X \sim X'$ in $\cC$. It follows that if $X \sim X'$ and $Y \sim Y'$, then on the one hand $\phi \sim \phi'$, and on the other hand $FX \sim FX'$ and $FY \sim FY'$, from which $F\phi \sim F\phi'$ by definition. $F: \cC \rarr \cD$ being given, having an equivalence relation on $\cC$ induces:
\begin{align}
	[F]: & \cC /\!\!\sim \,\,\rarr \cD /\!\!\sim \nonumber \\
	& [X] \mapsto [FX] \nonumber \\
	& ([\phi]: [X] \rarr [Y]) \mapsto ([F \phi]: [FX] \rarr [FY]) \nonumber
\end{align}
where we say we have a morphism $[\phi]: [A] \rarr [B]$ if we can exhibit a representative $\phi: A \rarr B$, and we say $[A] \rarr [B]$ is of type $\Lambda$ if there is a representative morphism of type $\Lambda$. Observe that if $\cD = \cM$ is a model category, and $\Lambda$ = C, W or F, meaning a cofibration, a weak equivalence or a fibration, then a same map $[A] \rarr [B]$ in $\cM /\!\!\sim$ can be of different types simultaneously, so we lose in detail. By definition it is clear that $\cM /\!\!\sim$ is still a model category, albeit a very weak one.\\

\section{$\mathbb{Z}$-schemes}
In a first time, we will apply the above formalism to $\cC = \Smk$, the category of smooth separated schemes of finite type over $\Spec(k)$, where $k$ is a field. Following \cite{FSV}, we consider an extension of this category using finite correspondences as morphisms of schemes, giving rise to a category SmCor(k), whose objects are smooth schemes of finite type over $k$, and whose morphisms are finite correspondences, essentially linear combinations of integral schemes (see \cite{FSV} for a complete definition).\\

There is a functor $[ \,] : \Smk \rarr \SmCor(k)$, that associates to any scheme $X$ an object $[X]$ of SmCor(k), and to any morphism $f: X \rarr Y$ in $\Smk$ its graph $\Gamma_f \subset X \times Y$. For morphisms of schemes $X \xrarr{f} Y$ and $Y \xrarr{g} Z$ whose composition is defined, the corresponding composition in SmCor(k) reads $[X] \xrarr{\Gamma_f} [Y] \xrarr{\Gamma_g} [Z]$, where the composition $\Gamma_g \circ \Gamma_f$ is defined as in \cite{FSV}. We generalize this construction by defining $\ZSmk$, the category whose objects are linear combinations of elements of $\Smk$(referred to as \textbf{$\bZ$-schemes}), and whose morphisms are appropriately chosen linear combinations of finite correspondences in a sense that we make precise presently.\\

Objects of $\ZSmk$ are linear combinations of smooth separated schemes of finite type over $\Spec(k)$, they are of the form $\sum m_i [X_i]$. In other terms $\ZSmk$ is a free abelian group on $\Smk$. We have an embedding $\iota(X) = 1 \cdot [X]$ from $\Smk$ to $\ZSmk$. A map $\phi$ from $\Smk$ to $\ZSmk$ naturally extends to a map $\Phi: \ZSmk \rarr \ZSmk$ as follows: $\Phi(\sum m_i [X_i]) = \sum m_i \phi([X_i])$. Indeed, there is a unique map $\phi_*$ that makes the following diagram commutative for $m \in \bZ$:
\beq
\xymatrix{
	\Smk \ar[d]_{\phi} \ar[r]^{m\iota} &\ZSmk \ar[d]^{\phi_*} \\
	\Smk \ar[r]_{m\iota} & \ZSmk 
} \nonumber
\eeq
that is $\phi_* \circ m \iota = m \iota \circ \phi$, that is $\phi_*(m[X]) = m \phi [X]$. This can formally be presented as saying $\phi_* = m \phi$, and this is the notation we will adopt. Of particular interest, observe that if we are looking at a morphism of schemes $X \xrarr{\phi} Y$, then this means $\phi_*: m[X] \rarr m[Y]$. Indeed, if we start from a morphism $\phi: [X] \rarr [Y]$, then $m\phi: [X] \rarr m[Y]$, so that $\phi_*  : m [X] \rarr m[Y]$. Henceforth, we will drop the $*$ notation. We also consider morphisms of the form $\sum m_i [X_i] \rarr \sum n_j[Y_j]$. Since each $[X_i]$ may map to different $[Y_j]$'s, $m_i$ will split as:
\beq
m_i = m_{i_{j_1}} + \cdots + m_{i_{j_{|i|}}} = \sum_{j \in J_i} \mij \nonumber
\eeq
where $|i|$ is the number of $[Y_j]$'s $[X_i]$ is mapping to, and $J_i$ is the subset of those indices $j$ of $J$ with a morphism $[X_i] \rarr [Y_j]$. It also means that:
\beq
n_j = \sum_{i \in I_j} \mij \nonumber
\eeq
where $I_j$ denotes the set of indices $i$ for which we have a morphism $[X_i] \rarr [Y_j]$. If $I_j$ is a singleton, then $n_j$ is simply not decomposed. Thus a morphism $\phi : \sum_{i \in I} m_i [X_i] \rarr \sum_{j \in J} n_j [Y_j]$ will decompose as follows:
\beq
\phi: \sum_{\substack{j \in J \\ i \in I_j}} \mij[X_i] \rarr \sum_{\substack{ j \in J \\ i \in I_j}} m_{i_j} [Y_j] \nonumber
\eeq
If we denote by $\phi_{ji}$ the restriction of $\phi$ to a map $ [X_i] \rarr [Y_j]$, then it is clear that we have:
\begin{align}
	\phi &= \sum_{\substack{ i \in I \\ j \in J_i}} \mij \phi_{ji} \nonumber \\
	&= \sum_{\substack{ j \in J \\ i \in I_j}} \mij \phi_{ji} \nonumber
\end{align}

Observe that each morphism $\phi_{ji}: [X_i] \rarr [Y_j]$ is in SmCor(k), i.e. it is really a finite correspondence. This makes $\phi$ a linear combination of finite correspondences, hence an element of $\ZSmk$ itself. \\

Consider another morphism:
\beq
\psi: \sum_{j \in J} n_j [Y_j] \rarr  \sum_{k \in K} p_k [Z_k] \nonumber
\eeq
with $\sum_{j \in J_k} \njk = p_k$. Here $\psi = \sum \njk \psi_{kj}$ with $\psi_{kj}: [Y_j] \rarr [Z_k]$. We will now define $\psi \circ \phi$. First consider:
\beq
m[X] \xrarr{m\phi} m[Y] \xrarr{m\psi} m[Z] \nonumber
\eeq
with $\phi: X \rarr Y$ and $\psi: Y \rarr Z$, then $m\psi \circ m \phi = m( \psi \circ \phi)$. Now given $k \in K$, if $j \in J_k$ we have a morphism:
\beq
\njk \psikj: \njk [Y_j] \rarr \njk [Z_k] \nonumber 
\eeq
If in addition $i \in I_j$, then we have:
\beq
\mijk [X_i] \xrarr{\mijk \phiji} \mijk [Y_j] \xrarr{\mijk \psikj} \mijk [Z_k] \nonumber
\eeq
where $\mijk$ is defined as follows. Write $(ijk)$ for $i_{j_k}$ for simplicity. We have $p_k = \sum_{j \in J_k} \njk$, and $\njk = \sum_{i \in I_j} m_{(ijk)}$. Then we can define our composition:
\beq
\sum m_i [X_i] \xrarr{\phi} \sum n_j [Y_j] \xrarr{\psi} \sum p_k [Z_k] \nonumber
\eeq
as:
\beq
\psi \circ \phi = \sum_{\substack{k \in K \\ j \in J_k}} \sum_{i \in I_j} m_{(ijk)} \psikj \circ \phiji \nonumber
\eeq

Thus defined, composition is clearly associative. In so doing, it helps to regard $\mijk$ as the coefficient of $[X_i]$ in the decomposition of $X$ for the composition $[X_i] \rarr [Y_j] \rarr [Z_k]$, hence $m_{(ijkl)}$ is the coefficient of $X_i$ in $[X_i] \rarr [Y_j] \rarr [Z_k] \rarr [W_l]$. Indeed, consider the following composition:
\beq
\sum m_i [X_i] \xrarr{\phi} \sum n_j [Y_j] \xrarr{\psi} \sum p_k [Z_k] \xrarr{\gamma} \sum q_l [W_l] \nonumber
\eeq
On the one hand:
\begin{align}
	\gamma \circ (\psi \circ \phi) &= \gamma \circ \sum m_{(ijk)} \psikj \circ \phiji \nonumber \\
	&= \sum_{\substack{l \in L \\ k \in K_l}} \sum_{\substack{j \in J_k \\ i \in I_j}} m_{(ijkl)} \gammalk \circ ( \psikj \circ \phiji) \nonumber
\end{align}
where:
\beq
q_l = \sum_{k \in K_l} \sum_{j \in J_k} \sum_{i \in I_j} m_{(ijkl)} \nonumber
\eeq
On the other hand:
\begin{align}
	(\gamma \circ \psi) \circ \phi &= (\sum_{\substack{l \in L \\ k \in K_l}} \sum_{j \in J_k} n_{(jkl)} \gammalk \circ \psikj) \circ \phi \nonumber \\
	&=(\sum_{\substack{l \in L \\ k \in K_l}} \sum_{\substack{j \in J_k\\ i \in I_j}} m_{(ijkl)} (\gammalk \circ \psikj)) \circ \sum m_{(ijkl)} \phiji \nonumber \\
	&= \sum_{\substack{l \in L \\ k \in K_l}} \sum_{\substack{j \in J_k\\ i \in I_j}} m_{(ijkl)} (\gammalk \circ \psikj) \circ \phiji \nonumber 
\end{align}
where we have $n_{(jkl)} = \sum_{i \in I_j} m_{(ijkl)}$. However, we clearly have $(\gammalk \circ \psikj) \circ \phiji = \gammalk \circ (\psikj \circ \phiji)$, hence associativity. About identity morphisms, if $\phi = \sum_{ j \in J, i \in I_j} \mij \phiji: \sum \mij [X_i] \rarr \sum \mij [Y_j]$, then a right inverse is provided by $\sum \mij id_{[Y_j]}$, and a left inverse by $\sum \mij id_{[X_i]}$.

\section{Sheaves on $\mathbb{Z}$-schemes}
The motivation for considering presheaves of sets on $\ZSmk$ is that the coproduct of schemes in $\Smk$ is not always well-defined, hence we have the same problem in $\ZSmk$ as well. A convenient way to fix this problem is to formally add colimits in $\ZSmk$ by considering presheaves of sets on $\ZSmk$, as already done in \cite{V} for SmCor($k$). We have a Yoneda embedding:
\begin{align}
	h: \ZSmk &\rarr \Set^{\ZSmk^{\op}} \nonumber \\
	X & \mapsto h_X = \Hom_{\ZSmk}(-,X) \nonumber
\end{align}

Now in $\PreSh(\ZSmk)$, for the sake of doing homotopy, we have well-defined pushouts.\\

Now another problem surfaces. As pointed out in \cite{V}, if $X = U \cup V$ is a covering of a scheme $X$ by two Zariski open subsets, the following diagram is a pushout in $\Smk$, hence in $\ZSmk$ as well:
\beq
\xymatrix{
	U \cap V \ar[d] \ar[r] &U \ar[d] \\
	V \ar[r] &X
} \nonumber
\eeq
but the corresponding square of representable presheaves:
\beq
\xymatrix{
	h_{U\cap V} \ar[d] \ar[r] &h_U \ar[d] \\
        h_V \ar[r] &h_X
}\nonumber
\eeq
is not necessarily a pushout in $\PreSh(\Smk)$, hence not one in $\PreSh(\ZSmk)$ either. This can be remedied by considering sheaves. Recall that a presheaf of sets $F: \cCop \rarr \Set$ is a sheaf in a topology $\tau$ on $\cC$ if for any covering family $\{f_{\alpha}: U_{\alpha} \rarr X\}$ in this topology, the following sequence is exact:
\beq
F(X) \rarr \prod_{\alpha} F(U_{\alpha}) \rightrightarrows \prod_{\alpha, \beta}F(U_{\alpha} \times_X U_{\beta}) \nonumber
\eeq
Thus it is clear that we will have to introduce a topology on $\ZSmk$. We will prove that sheaves, which are contravariant functors from $\ZSmk$ to $\Set$ map certain pushout squares to cartesian squares. Those pushout squares are referred to as elementary distinguished squares in \cite{V}, \cite{FSV}, \cite{ORV}, \cite{VM}, and they are defined in terms of etale morphisms of $\bZ$-schemes. Those require the notion of sheaves on $\mathbb{Z}$-schemes.\\

\subsection{Sheaves on $\mathbb{Z}$-schemes}
For $X = \sum m_i [X_i]$ an object of $\ZSmk$, a sheaf $\cF$ on $X$ decomposes as $\cF = \times_i \cF_i$ where each $\cF_i$ is a sheaf on $m_i[X_i]$. However for sheaves of abelian groups, addition is well-defined for a product of sheaves if we consider their tensor product instead, so $\cF  = \otimes_i \cF_i$. Now it suffices to define sheaves at objects of the form $m[X]$ for $m \in \mathbb{Z}$ and $X \in \Smk$.\\

For $X \in \ZSmk$ of the form $X = m [X]$ presently we adopt the notation $\uX$ for $[X]$ in such a manner that $ X = m \uX$, and in the same fashion, if $\uF$ is a sheaf of abelian groups on $\Smk$, or even a presheaf to be more general, then one can define $\cF = m \uF$ as a presheaf on $\ZSmk$ according to the following commutative diagram:
\beq
\xymatrix{
	\uU \ar[d]_{ \times m} \ar[r]^{\uF} & \uF(\uU) \ar[d]^{\times m} \\
U = m \uU \ar@{.>}[r]^-{\cF = m \uF} & \cF(U) = \cF(m \uU) =  m\uF(\uU) 
} \nonumber
\eeq
We now define sheaves of $\cO_X$-modules on $X = \sum m_i [X_i] \in \ZSmk$. Let $\cF$ be a sheaf of abelian groups on $X$, $\cF = \otimes \cF_i|_{m_iX_i} = \otimes_i m_i \uFi|_{X_i}$, with $\uFi$ sheaf on $X_i$, $\cO_X = \otimes m_i \cO_{X_i}$. A sheaf of $\cO_X$-modules is a sheaf $\cF$ on $X$ such that for each open set $U = \sum_{i \in I} m_i [U_i]$, $U_i$ open in $X_i$, the group $\cF(U) = \otimes m_i \uFi(U_i)$ is a $\otimes m_i \cO_{X_i}(U_i)$ -module, that is $\uFi(U_i)$ is a $\cO_{X_i}(U_i)$-module for all $i \in I$, and for each inclusion of open sets $V \subset U$ in $\ZSmk$, the indexed restriction homomorphisms are compatible with the module structure.\\

\subsection{Flatness} \label{flat}
Following \cite{H} for the case of $\Smk$, that we generalize to $\ZSmk$, let $\phi: X = \sum m_i [X_i] \rarr Y = \sum n_j [Y_j]$ be a morphism in $\ZSmk$, $\cF$ a $\cO_X$-module. Let $x = \sum m_i x_i \in X$ with $x_i \in X_i$ for all $i \in I$. Since we have a decomposition $\phi  = \sum_{j \in J, i \in I_j} m_{i_j}\phi_{ji}$, with $\phi_{ji}: [X_i] \rarr [Y_j]$, it follows that one can write $x = \sum_{j \in J, i \in I_j} m_{i_j} x_i$ in such a manner that $m_{i_j} x_i$ maps to some $m_{i_j} y_{ji}$ in $Y_j$ for all $i \in I_j$, that is $\phi_{ji}(x_i) = y_{ji}$. Letting $y = \sum_{j \in J, i \in I_j} m_{i_j} y_{ji}$ in $Y$, one says $\cF$ is flat over $Y$ at $x$ if $\cF_x = \otimes_{j \in J} \otimes _{i \in I_j} m_{i_j} (\underline{\cF}_i)_{x_i}$ is a flat $\cO_y = \otimes_{j \in J} \otimes_{i \in I_j} m_{i_j} \cO_{y_{ji}}$-module, where we consider $(\uFi)_{x_i}$ a $\cO_{y_{ji}}$-module via the natural maps $\phi_{ji}^{\#}: \cO_{y_{ji}} \rarr \cO_{x_i}$. Then one says $\cF$ is flat over $Y$ if it is flat at every point of $X$ and one says $X$ is flat over $Y$ if $\cO_X$ is. Observe that if $Y = \Spec (k)$, if each $X_i$ is flat, so is $X = \sum m_i[X_i]$.\\

\subsection{Sheaf of relative differentials}
Remember from \cite{H} that the sheaf of relative differentials $\Omega_{X/Y}$ for a morphism of schemes $f: X \rarr Y$ is defined as $\Delta^*(\cI/\cI^2)$ where $\Delta: X \rarr X \times_Y X$ is the diagonal map and $\cI$ is the sheaf of ideals of $\Delta X$ on $W$, open subset of $X \times_Y X$, $\Delta X$ closed subscheme thereof. For $\phi: X = \sum m_i [X_i] \rarr Y = \sum n_j [Y_j]$, we have the usual decomposition $\phi = \sum_{j \in J, i \in I_j} \mij \phi_{ji}$, so that:
\beq
X \times_Y X = \sum_{\substack{j \in J \\ i \in I_j}} \mij [X_i] \times_{[Y_j]} [X_i] \supset W \supset \Delta(X) \nonumber
\eeq
with $\Delta(X) = \sum_{j \in J, i \in I_j} \mij \Delta_{ji} [X_i]$, where $\Delta_{ji}: [X_i] \rarr [X_i] \times_{[Y_j]} [X_i]$ is the diagonal map, in such a manner that if $\cI_{ji}$ is the sheaf of ideals of $\Delta_{ji}[X_i]$ on $ W \cap [X_i] \times_{[Y_j]} [X_i]$, then we have:
\beq
\cI = \otimes_{\substack{j \in J \\ i \in I_j}} \mij \cI_{ji} \nonumber
\eeq
so that:
\beq
\cI /\cI^2 = \otimes \mij \cI_{ji}/\cI_{ji}^2 \nonumber
\eeq
Now recall (\cite{H}) that for $f: X \rarr Y$ a morphism of schemes, for $\cG$ a sheaf on $Y$, $f^* \cG = f^{-1} \cG \otimes_{f^{-1}\cOY} \cOX$ is a sheaf of $\cOX$-modules, where $f^{-1} \cG$ is the sheaf associated to the presheaf $U \mapsto \lim_{V \supseteq f(U)} \cG(V)$. Here $\Delta^* (\cI/\cI^2) = \Delta^{-1} \cI / \cI^2 \otimes_{\Delta^{-1} \cO_{X \times_Y X}} \cOX$, where $\Delta^{-1} \cI / \cI^2$ is the sheaf associated with $U = \sum m_i [U_i] \mapsto \lim_{V \supseteq \Delta U} \cI/ \cI^2(V)$, with $\Delta U = \sum \mij \Dji [U_i] \subset \sum \mij V_{ji} = V$, and $V_{ji}$ is open in $[X_i] \times_{[Y_j]} [X_i]$. We also have:
\begin{align}
	\cI / \cI^2(V) &= \big( \otimes \mij \cIji/ \cIji^2 \big) (\sum \mij V_{ji}) \nonumber \\ 
	&= \otimes \mij \cIji / \cIji^2 (V_{ji}) \nonumber
\end{align}
Indeed recall that if $\cF = m \underline{\cF}$, $\cF(m[X]) = m \underline{\cF}[X]$. Note that once $X$ is fixed, so are the coefficients $m_i$, hence so are the $\mij$'s as well once $Y$ is fixed, so the above decomposition does not depend on $V$. It follows:
\begin{align}
	\lim_{V \supseteq \Delta U} \cI / \cI^2(V) &= \lim_{\sum \mij V_{ji} \supseteq \sum \mij \Dji [U_i]} \otimes \mij \cIji / \cIji^2(V_{ji}) \nonumber \\
	 &= \otimes \mij \lim_{V_{ji} \supseteq \Dji [U_i]} \cIji / \cIji^2 ( V_{ji}) \nonumber 
\end{align}
so that $\Delta^{-1} \cI / \cI^2 = \otimes \mij \Dji^{-1} \cIji / \cIji^2$, as well as $\Delta^{-1} \cO_{X \times_Y X} = \Delta^{-1} (\otimes \mij \cO_{[X_i] \times_{[Y_j]} [X_i]}) = \otimes \mij \Dji^{-1} \cO_{[X_i] \times_{[Y_j]} [X_i]}$, so that:
\begin{align}
	\Delta^* ( \cI / \cI^2) &= \Delta^{-1} \cI / \cI^2 \otimes_{ \Delta^{-1} \cO_{X \times_Y X}} \cOX \nonumber \\
	&= \otimes \mij \big( \Dji^{-1} \cIji / \cIji^2 \otimes_{\Dji^{-1} \cO_{[X_i] \times_{[Y_j]} [X_i]}} \cOXi \big) \nonumber \\
	&=\otimes \mij \Dji^* \cIji / \cIji^2 \nonumber \\
	&= \otimes \mij \Omega_{X_i / Y_j} \nonumber \\
	&= \otimes \Omega_{\mij [X_i] / \mij [Y_j]} = \Omega_{X/Y} \nonumber
\end{align}
where in going from the first line to the second, we used $m\underline{\cF} \otimes_{m \cO_Y} m \underline{\cG} = m \underline{\cF}\otimes_{\cO_Y} \underline{\cG}$ since $ m \underline{\cF} \otimes m \underline{\cG} = m(\underline{\cF} \otimes \underline{\cG})$ is a sheaf of $m \cO_Y$-modules.

\subsection{Etale maps}
Recall that an etale map $f: X \rarr Y$ is a smooth map of relative dimension zero, which in $\Smk$ means $f$ flat, for any irreducible components $X' \subset X$ and $Y' \subset Y$ such that $f(X') \subset Y'$, we have  $\deem(X') = \deem(Y')$, and $\deem_{k(x)}(\Omega_{X/Y} \otimes k(x)) = 0$ for any point $x \in X$. We generalize these notions to $\ZSmk$. We first define what it means to be a morphism of finite type in $\ZSmk$, since smooth of relative dimension zero subsumes of finite type. We simply define a morphism $\phi: \sum m_i [X_i] \rarr \sum n_j [Y_j]$ in $\ZSmk$ to be of finite type if in the decomposition $\phi = \sum_{j \in J, i \in I_j} \mij \phi_{ji}$, each of the morphisms $\phi_{ji}: [X_i] \rarr [Y_j]$ is of finite type itself, for $j \in J$ and $i \in J_j$. We define $\phi: X \rarr Y$ in $\ZSmk$ to be etale in exactly the same fashion that it was defined in $\Smk$, namely $\phi$ flat, for any irreducible components $X'$ of $X$ and $Y'$ of $Y$, if $\phi(X') \subset Y'$, then $\deem X' = \deem Y'$, and for any $x \in X$, $\deem ( \Omega_{X/Y} \otimes k(x)) = 0$. Flatness has been defined in section 4.2. For the dimensional statement on irreducible components, consider $X' = \sum m_i [X_i'] \subset X$ and $Y' = \sum n_j [Y_j'] \subset Y$. We consider a morphism $\phi: X \rarr Y$ which has the usual decomposition $\phi = \sum \mij \phi_{ji}$ so that we can actually write $X' = \sum_{j \in J, i \in I_j} \mij [X_i']$ and $Y' = \sum_{j \in J, i \in J_i} \mij [Y_j']$. However for irreducible components, we just consider individual such terms, $X'$ is of the form $[X_i']$ for some $i \in I_j$, and $Y' = [Y_j']$. Then the condition $\phi[X_i'] \subset [Y_j']$ reads $\phi_{ji} [X_i'] \subset [Y_j']$. Then $\deem(X_i') = \deem(Y_j')$ for all such choices if the $\phiji$'s satisfy this dimensional statement on irreducible components i.e. $\phi = \sum \mij \phiji$ satisfies the dimensional statement on irreducible components if all of the $\phiji$ do. Finally we generalize the dimensional statement involving the sheaf of relative differentials. The local ring $\cO_x = \otimes m_i \cO_{x_i}$ has $\gm_x = \otimes m_i \gm_{x_i}$ for maximal ideal, with residue field $k(x) = \otimes m_i \cO_{x_i}/m_i \gm_{x_i} = \otimes m_i\cO_{x_i}/\gm_{x_i}$. Now since we consider a morphism $\phi: X = \sum m_i [X_i] \rarr Y = \sum n_j [Y_j]$ then we have to consider the decomposition $\cO_x = \otimes_{j \in J, i \in I_j} \mij \cO_{x_i}$, from which it follows $k(x) = \otimes_{j \in J, i \in I_j} \mij \cO_{x_i}/\gm_{x_i}$. Then we have:
\begin{align}
	\Omega_{X/Y} \otimes k(x) &= \otimes_{\substack{j \in J \\ i \in I_j}} \mij \big(\Omega_{X_i/Y_j} \otimes \cO_{x_i}/\gm_{x_i}\big)  \nonumber \\
	&= \otimes \mij \big( \Omega_{X_i/Y_j} \otimes k(x_i) \big) \nonumber
\end{align}
Hence:
\beq
\dim \big( \Omega_{X/Y} \otimes k(x) \big)= \sum_{\substack{j \in J \\ i \in I_j}} \mij \dim(\Omega_{X_i/Y_j}\otimes k(x_i)) = 0 \nonumber 
\eeq
if each summand is zero. We have shown:
\begin{et} \label{et}
	$\phi = \sum \mij \phiji$ is etale if and only if each $\phiji$ is etale.
\end{et}

\subsection{Elementary distinguished squares}
Following \cite{V}, \cite{VM}, \cite{ORV}, we define an elementary distinguished square in $\ZSmk$ to be a square of the form:
\beq
\xymatrix{
	p^{-1}(U) \ar[d] \ar[r] &V \ar[d]^p \\
	U \ar[r]_{\psi} &X
} \nonumber
\eeq
where $p$ is an etale morphism of $\mathbb{Z}$-schemes, $\psi$ is an open embedding, and $p^{-1}(X-U) \cong X-U$. Observe that $\psi$ being an open embedding implies that if $X = \sum m_i [X_i]$, $U \cong \sum m_i [U_i]$, where $U_i$ is an open subset of $X_i$.\\ 

We define elementary distinguished squares in this section, since they deal with morphisms of $\mathbb{Z}$-schemes. However it is in the definition of sheaves on $\ZSmk$ that such squares are important, and we define those next.\\

\section{Sheaves on $\ZSmk$}
\subsection{Sheaves}
Since representables presheaves generate presheaves \cite{MML}, we will deduce properties of sheaves on $\ZSmk$ from those of representable presheaves. We first consider $\cF = \Hom(-,\sum n_j[Y_j])$. We have:
\begin{align}
	\cF(\sum m_i [X_i]) &=\Hom(\sumI m_i [X_i], \sumJ n_j [Y_j]) \nonumber \\
	&= \oplus_{i \in I} \Hom(m_i [X_i], \sumJ n_j [Y_j]) \nonumber \\
	&=\oplus_{i \in I} \cF(m_i [X_i]) \nonumber
\end{align}
hence for sheaves $\cF$ on $\ZSmk$, we also have:
\beq
\cF(\sum m_i [X_i]) = \oplus_{i \in I} \cF(m_i [X_i]) \nonumber
\eeq

A presheaf $\cF$ is a sheaf for a Grothendieck topology on $\ZSmk$ if for any covering $\{ \fa: \Ua \rarr X \}$ in $\ZSmk$ for this topology, we have an equalizer:
\beq
\cF(X) \rarr \prod_{\alpha} \cF(\Ua) \rightrightarrows \prod_{\alpha, \beta} \cF(\Ua \times_X \Ub) \nonumber
\eeq

\subsection{Nisnevich topology on $\ZSmk$}
We define a Nisnevich covering in $\ZSmk$ as in \cite{V}, \cite{VM}, \cite{ORV}, to be a finite family of etale morphisms $\{\fa: \Ua \rarr X\}$  such that for every $x \in X$, there is some $\alpha$, there is some $u \in \Ua$ mapping to $x$ such that $k(u) \cong k(x)$. On $\ZSmk$ this reads as follows. $\{\fa: \Ua \rarr X \}$ is a Nisnevich covering if for any $x = \sum m_j x_j$ in $X$ with $x_j \in X_j$, there is some $\Ua = \sum_{i \in \Ia} \muai [\Uai] = \sum_{j \in J, i \in \Iaj} \muaij [\Uai]$, there are $u_{\alpha i} \in U_{\alpha i}$  for all $i \in \Iaj $, with $\faji: u_{\alpha i} \rarr x_j$ (subject to $\sum_{i \in \Iaj} \muaij = m_j$), such that $k(u_{\alpha i}) \cong k(x_j)$, and this for all $j \in J$. The morphisms $\fa$ being etale means that each $\faji: \Uai \rarr X_j$ is etale. 

\begin{NisCov}
$\{\fa: \Ua \rarr X, \; \alpha \in A \}$ is a Nisnevich covering if for all $j \in J$, $\{\fa: \Ua \rarr m_j [X_j] , \; \alpha \in A_j\}$ is a Nisnevich covering, and $A = \cap_{j \in J} A_j$. 
\end{NisCov}
\begin{proof}
To have a Nisnevich covering over $m_j[X_j]$ means for any $x_j \in X_j$, there is some $\alpha \in A_j$, there is some $\ua \in \Ua$ mapping to $m_j x_j$, such that $k(\ua) \cong k(m_j x_j)$. Precisely, this means there is some $\ua = \sum_{i \in \Iaj} \muaij \uai \in \Ua = \sum_{i \in \Iaj} \muaij [\Uai]$, with each $\uai$ mapping to $x_j$ for all $i \in \Iaj$, with $k(\uai) \cong k(x_j)$, subject to $\sum_{i \in \Iaj} \muaij = m_j$.	Now if $x = \sum m_j x_j \in X$, assuming the hypothesis of the lemma, there is some $ \alpha \in  \cap_{j \in J} A_j$ (after possible reindexing of the $\Uai$'s), there is some $u_{\alpha i} \in \Uai$ for any $i \in \Iaj$ with $u_{\alpha i} \rarr x_j$, giving the decompositions:
\begin{align}
	k(x) &= \otimes_{\substack{j \in J \\ i \in I_j}} \muaij \cO_{x_j}/\muaij \gm_{x_j} = \otimes_{\substack{j \in J \\ i \in I_j}} \muaij \cO_{x_j} / \gm_{x_j} = \otimes \muaij k(x_j) \nonumber \\
	k(\ua) &= \otimes_{\substack{j \in J \\ i \in I_j}} \muaij \cO_{\uai}/\muaij \gm_{\uai} = \otimes_{\substack{j \in J \\ i \in I_j}} \muaij \cO_{\uai} / \gm_{\uai} = \otimes \muaij k(\uai)\nonumber 
\end{align}
with $u_{\alpha} = \sum \muaij u_{\alpha i}$. Thus we see that $k(x) \cong k(u_{\alpha})$ if $k(x_j) \cong k(u_{\alpha i})$ for each $j \in J$ and $i \in \Iaj$.In other terms $\{\fa: \Ua \rarr X, \; \alpha \in A \}$ is a Nisnevich covering if for all $j \in J$, $\{ \fa: \Ua \rarr m_j [X_j] , \, \alpha \in A_j \}$ is a Nisnevich covering, and $A = \cap_{j \in J} A_j$. 
\end{proof}	

For later purposes, denote by $(a_1X \times_{mX} \cdots \times_{mX} a_nX)_{\Sigma}$ the limit $a_1X \times_{mX} \cdots \times_{mX} a_nX$, subject to $\sum a_i = m$, and denote by $(a_1X \times_{mX} \cdots \times_{mX} a_nX)^{\Delta}$ the diagonal of the limit $a_1X \times_{mX} \cdots \times_{mX} a_nX$. We can then represent a Nisnevich covering over $m_j [X_j]$ as a Nisnevich covering $\{ \prod \muaij \faji: \prod \muaij \Uai \rarr (\mu_{\alpha i_1}X_j\times_{m_jX_j} \cdots \times_{m_jX_j} \mu_{\alpha i_k} X_j)^{\Delta}_{\Sigma} \}$ where $k = |\Iaj|$.\\

We denote by $\Sh(\ZSmk, \Nis)$ the category of sheaves on $\ZSmk$ equipped with the Nisnevich topology.

\subsection{Elementary distinguished squares}
We now prove a generalization of Proposition 3.1.4 of \cite{VM}, which in the present situation would read as follows:
\begin{ShSqu} \label{ShSqu}
	A presheaf on $\ZSmk$ is a sheaf if and and only if it maps every elementary distinguished square in $\ZSmk$ to a cartesian square.
\end{ShSqu}
	The proof is identical in form to \cite{VM}, and differs only in the fact that we work in $\ZSmk$, not $\Smk$, hence we have to deal with hybrid/local indexed terms, which does not make the proof any different in spirit, but there are technicalities we have to keep track of.\\

	That a sheaf on $\ZSmk$ maps elementary distinguished squares to cartesian squares follows from the original proof of \cite{VM} due to its formality. Vice-versa, suppose now a presheaf $\cF$ on $\ZSmk$ maps elementary distinguished squares to cartesian squares. We aim to show it is a sheaf. In other terms if $\cU = \{ \fa: \Ua \rarr X \}$ is a Nisnevich covering, we want:
	\beq
	\cF(X) \rarr \prod_{\alpha} \cF(\Ua) \rightrightarrows \prod_{\alpha, \beta} \cF(\Ua \times_X \Ub) \nonumber
	\eeq
	to be exact. To do so we define a splitting sequence for $\cU$ in exactly the same manner that it was initially introduced in \cite{VM}, but obviously adapted to our setting. We first need to prove that if $\cU$ is a Nisnevich covering, it admits a splitting sequence. This means we have to first define rational sections in $\ZSmk$.\\

\subsection{Rational sections of $\ZSmk$}
We say $X = \sum m_i [X_i]$ is Noetherian if each $X_i$ is Noetherian in $\Smk$ for all $i$. We generalize to $\ZSmk$ the definition of rational maps as initially introduced in \cite{G}. Let $X = \sum m_i [X_i]$ and $Y = \sum n_j [Y_j]$ be two objects of $\ZSmk$. Let $U = \sum m_i [U_i]$ and $V = \sum m_i [V_i]$ be two open subsets of $X$. Then $f: U \rarr Y$ and $g: V \rarr Y$ are said to be equivalent if they coincide in an open dense subset of $U \cap V$, of the form $\sum m_i [W_i]$, with $W_i$ an open dense subset of $U_i \cap V_i$, i.e. if $f|_{U_i}$ and $g|_{V_i}$ agree on $W_i$ for all $i$. Then one defines a rational map $X \rarr Y$ in $\ZSmk$ to be an equivalence class of morphisms of dense open subsets $X' = \sum m_i [X'_i]$ of $X$ into $Y$, with $X'_i$ dense open subset of $X_i$ for each $i$. To be specific, a rational map from $X= \sum m_i [X_i]$ to $\Ua = \sum \muaj [\Uaj]$ has for representation:
\beq
\sum_{i \in I}(\sum_{j \in \Jai} \mij [X_i']) \rarr \sum_{i \in I} ( \sum_{j \in \Jai} \mij [\Uaj]) \nonumber
\eeq
where $X_i'$ is a dense open subset of $X_i$ for all $i \in I$ and $\sum_{i \in I_j} \mij = \muaj$, $m_i = \sum_{j \in \Jai} \mij$. The above map is given on each $X_i'$ by $\sum_{j \in \Jai} \mij [X_i'] \rarr \sum_{j \in \Jai} \mij [\Uaj]$, so a rational map is of the form $\sum_{i \in I, j \in \Jai} \mij \raji$ where all $\raji: [X_i] \rarr [\Uaj]$ for $j \in \Jai$ are rational maps, simultaneously over the same open set $X_i'$ for $i$ fixed, that is giving a rational map on $[X_i]$ is equivalent to giving a map:
	\beq
	\big( m_{i_1}X_i' \times_{m_iX_i'} \cdots \times_{m_iX_i'}m_{i_k} X_i'\big)^{\Delta}_{\Sigma} \rarr \prod_{j \in \Jai} \mij \Uaj \nonumber
	\eeq
Now a rational section of $\Ua \rarr X$ is a rational map $X \rarr \ \Ua = \sum_{i \in I, j \in \Jai} \muaij [\Uaj]$ which is also a section, i.e. a map $\sum_{i \in I, j \in \Jai} \mij \saji$ where each $\saji$ is a rational map, and a section, so that $\sum_{i \in I, j \in \Jai} \mij \paij \circ \saji =  id_X$, with $\paij: [\Uaj] \rarr [X_i]$. \\

\subsection{Construction of rational sections of Nisnevich covers}
Observe that in the initial Nisnevich covering of $X = \sum m_j [X_j]$, we have morphisms $\fa: \Ua \rarr X$, with each $\Ua = \sum_{j \in J, i \in \Iaj} \muaij [\Uai]$. If we write $U = \coprod_{\alpha} \Ua$, we have:
\beq
U = \coprod_{\alpha} \sum_{\substack{j \in J \\ i \in \Iaj}} \muaij [\Uai] = \sum_{j \in J} \coprod_{\alpha} \sum_{i \in \Iaj} \muaij [\Uai] \nonumber
\eeq
and the collection of morphisms $\coprod_{\alpha} \sum_{i \in \Iaj} \muaij [\Uai] \rarr m_j [X_j]$ forms a Nisnevich covering of $m_j [X_j]$. Indeed, $\fa = \sum_{ij} \muaij \faji$ etale implies $\sum_i \muaij \faji$ etale for $\alpha \in A$, so $\coprod_{\alpha} \sum_i \muaij \faji$ is etale by Lemma ~\ref{et}. It follows $\{ \fa: \Ua \rarr X \}$ is a Nisnevich covering implies $\{ \coprod_{\alpha} \sum_i \muaij \Uai \rarr m_j [X_j] \}$ is a Nisnevich covering. We just drop the index $j$ and call the above coproduct $U$. Observe, as pointed out in \cite{VM}, that to give a rational map from $mX$ to $U$ is equivalent to giving one on each irreducible component of $X$, so we might as well assume $X$ to be irreducible. Now we apply Lemma 1.5 of \cite{Ho} to $U=\coprod_{\alpha} \sum_{i \in \Ia} \muai [\Uai] \rarr mX$. Note that we can write this coproduct as $\sum_{\alpha, i \in \Ia} \muai [\Uai]$. Let $x$ be the generic point of $X$, let $\alpha$ be an index such that there is some $u \in \Ua$ over $x$. After reindexing, write $U_i = \muai [\Uai] \coprod$ possible other schemes, none of which is of the form $\mu_{\alpha j} [\Uaj]$ for $j \neq i$. Let $\Ia'$ be the indexing set for those $i$'s. Let $p_{\alpha i} : \Uai \rarr X$. Then each $p_i = \muai \pai \coprod \cdots : U_i \rarr \muai X$ is etale, of finite type, completely decomposed in the sense of Hoyois. It follows it has a rational section $\sigma_i$ for all $i$, hence so does $\prod p_i: \prod U_i \rarr \mu_{\alpha 1} X \times \cdots \times \mu_{\alpha |I_{\alpha}'|}X$. For our rational section we take:
\beq
(\mu_{\alpha 1}X' \times_{mX'} \cdots \times_{mX'} \mu_{\alpha |\Ia'|}X')^{\Delta}_{\Sigma} \xrarr{\prod_{i \in \Ia'} \muai \sigma[\pai]} \prod_{i \in \Ia'} U_i \nonumber
\eeq
as representative, $X'$ dense open subset of $X$, where $\sigma_i = \muai \sigma [\pai]$, as constructed in the previous subsection.

\subsection{Existence of splitting sequences for Nisnevich covers}
We now construct a splitting sequence for $\cU$ over $mX$. We have argued $p = \sum p_i: U = \sum U_i \rarr mX$ has a rational section $\sum_{i \in \Ia'} \muai \sigma [\pai]$, so there is some dense open subset $X'$ of $X$ such that we have a map $\sigma: mX' \rarr U$, section of $p^{-1}(mX') \rarr mX'$, $\sigma = \sum_{i \in \Ia'} \muai \sigma[\pai]$, $p = \sum_{\alpha, i \in \Ia'} \muai \pai$, $p \circ \sigma = mid_{X'}$. The rest of the construction is identical to that of \cite{VM} or \cite{Ho}. This proves that we have a splitting sequence for $\cU$. With this in hand, we can now finish the proof of Proposition ~\ref{ShSqu}:
\begin{dir2}
If a presheaf on the Nisnevich site $\ZSmk$ maps elementary distinguished squares to cartesian squares, it is a sheaf. 
\end{dir2}
\begin{proof}
	Let $\cU = \{ U_i \rarr mX \}$ be a Nisnevich covering of $mX$. The reasoning will be the same as in \cite{VM}, or \cite{Ho}. The only addition we bring here is the index notation to keep track of the components. Let $mX = mZ_0, \cdots, mZ_{n+1} = \eset$ a splitting sequence of minimal length for $\cU$, which exists as we have just shown. Choose a splitting for the morphism $p^{-1}(mZ_n) \rarr mZ_n$, whose existence is guaranteed by the previous subsection. This means picking a rational section $\sigma = \sum_{i \in \Ia'}  \sigma_i$. Since each $p_i : U_i \rarr \muai X$ is etale, we have a decomposition $p_i^{-1}(\muai Z_n) = \text{Im}( \sigma_i) \coprod Y_i$, with $\sigma_i = \muai \sigma[\pai]$, with $Y_i$ a closed subset of $U_i$. Then as in cite \cite{VM}, we let $W = X - Z_n$, $V_i = U_i - Y_i$, $V = \prod_{i \in \Ia'} V_i$. We then claim that $mW$ and $V$ form elementary distinguished squares over $mX$, and that $\cU \times_{mX} mW \rarr mW$ is a Nisnevich covering of $mW$ with a splitting sequence of length $n-1$. For the first claim, we have the following elementary distinguished square as a classical result:
\beq
\xymatrix{
	p_i^{-1}(\muai W) \ar[d] \ar[r] & U_i - Y_i \ar[d]^{p_i} \nonumber \\
\muai W = \muai(X-Z_n) \ar[r]^-{\muai \psi_i}  & \muai X
}
\eeq
with $p_i$ etale, $\psi_i$ open immersion. This is an elementary distinguished square as argued in \cite{VM}. It follows that the following square is also an elementary distinguished square:
\beq
\xymatrix{
	p^{-1}(mW) \ar[d] \ar[r] &\prod(U_i - Y_i)=V \ar[d]^{ \prod p_i = p} \\
	mW \cong \prod_{i \in \Ia} \muai W \ar[r]^-{\prod \muai \psi_i} &mX \cong (\mu_{\alpha 1} X \times_{mX} \cdots \times_{mX} \mu_{\alpha k} X)^{\Delta}_{\Sigma}
} \nonumber
\eeq
About the second point, $\{ p_i: U_i \rarr X \}$ is a Nisnevich covering, so by \cite{VM}, $\cU \times_{\muai X} \muai W \rarr \muai W$ is a Nisnevich covering, from which it follows that $\cU \times_{mX} mW \rarr mW$ is a Nisnevich covering.
\end{proof}

\section{Powered topologies}
We now define a notion of layered morphism, and a corresponding notion of layered (or powered) topology. We define this iteratively. Let $\cC_{N-1}$ be a category with objects $\chi_N$ of some type $\Lambda_N$. With this terminology, $\cC = \cC_0$ is our initial category, with objects $\cX_1$ of type $\Lambda_1$. Suppose each object $\cX_N$ of $\cC_{N-1}$ has some internal structure, and can be regarded as being made up of objects $\cX_{N+1}$ of type $\Lambda_{N+1}$. Categorify each such object $\cX_N$ in such a manner that its objects as a category are its constituting elements, and its morphisms are maps $\cX_{N+1}' \rarr  \cX_{N+1}$ between objects of $\cX_N$, if such maps exist. A \textbf{layered morphism} is any commutative diagram of the form:
\beq
\xymatrix{
	\cX_N \ar[r] &\cY_N \\
\cX_{N+1} \ar@{.>}[u] \ar[r] & \cY_{N+1} \ar@{.>}[u]
} \label{1p9}
\eeq
with possible additional lower layers defined as in:
\beq
\xymatrix{
	\cX_N \ar[r] &\cY_N \\
	\vdots \ar@{.>}[u] & \vdots \ar@{.>}[u] \\
        \cX_{N+p} \ar@{.>}[u] \ar[r] & \cY_{N+p} \ar@{.>}[u]
} \nonumber
\eeq
where in \eqref{1p9}, $\cX_N$ and $\cY_N$ are categories, $\cY_{N+1}$ is an object of $\cY_N$, $\cX_{N+1}$ is an object of $\cX_N$ and we have well-defined maps $\cX_N \rarr \cY_N$ and $\cX_{N+1} \rarr \cY_{N+1}$. Define $\cC^{[N,N+p]}$ as the category with objects of the form $(\cX_{N+p} \hookrightarrow \cdots \hookrightarrow \cX_N)$, with morphisms layered morphisms such as the one above. Identity and composition are obvious. A presheaf on $\cC^{[N,N+p]}$ is a functor $F: \cC^{[N,N+p] \op} \rarr \text{PoSet}$ that maps $ \cX_{N+p} \rarr \cdots \rarr \cX_N$ to $F_{N+p} \cX_{N+p} \larr \cdots \larr F_N \cX_N$. More generally, a functor from $\cC^{[N,N+p]}$ to $\cD^{[N,N+p]}$ is a map $F: \cC^{[N,N+p]} \rarr \cD^{[N,N+p]}$ with $F( \cX_{N+p} \hookrightarrow \cdots \hookrightarrow \cX_N) = F_{N+p} \cX_{N+p} \hookrightarrow \cdots F_N \cX_N$. Strictly speaking, $F_{N+p}$ is a functor on $\cX_{N+p-1}$, which could be different from a functor on $\cY_{N+p-1}$, for which we still use the notation $F_{N+p}\cY_{N+p}$, but categories are assumed by construction to be levelwise of a same type. Thus functors in this setting are understood levelwise not as functors from one category to another, but from one type of categories to another type of categories. Hence:
\beq
\xymatrix{
	\cX_N \ar[r]^{\phi_N} & \cY_N \\
	\cX_{N+p} \ar@{.>}[u] \ar[r]_{\phi_{N+p}} & \cY_{N+p} \ar@{.>}[u]
} \nonumber
\eeq
is mapped to:
\beq
\xymatrixcolsep{5pc}
\xymatrix{
	F_n\cX_n \ar[r]^{F_N \phi_N} & F_N \cY_N \\
	F_{N+p} \cX_{N+p} \ar@{.>}[u] \ar[r]_-{F_{N+p} \phi_{N+p}} & F_{N+p} \cY_{N+p} \ar@{.>}[u]
} \nonumber
\eeq
under $F$. We have $F(id_{\cX_{N+p} \hookrightarrow \cdots \hookrightarrow \cX_N}) = id_{F(\cX_{N+p} \hookrightarrow \cdots \hookrightarrow \cX_N)}$, since each $F_i$ is a functor on types, and $F(g \circ f) = F(g) \circ F(f)$ as well, represented as $F((g \circ f)_{N+p}, \cdots, (g \circ f)_N) = (F(g)_{N+p} \circ F(f)_{N+p}, \cdots, F(g)_N \circ F(f)_N)$.\\

Once that is defined, we can define what we call a layered topology. If $N = 1$, maps $\cX_1 \rarr \cY_1$ are in $\cC_0$. If one categorifies $\cX_1$ and $\cY_1$, $\cX_1 = \{ \cX_2 \}$ and $\cY_1 = \{ \cY_2 \}$, in writing:
\beq
\xymatrix{
	\cX_1 \ar[r] & \cY_1 \\
	\cX_2 \ar@{.>}[u] \ar[r] & \cY_2 \ar@{.>}[u]
} \nonumber
\eeq
the bottom map is no longer in $\cY_1$. Hence if we regard such a bottom map as an element of a covering of $\cY_2$, necessarily coverings, hence layered topologies, must be interpreted in a looser sense. We formalize this: suppose our categories admit pullbacks. Define a basis for a \textbf{loose topology} on $\cY_N$ for $N \geq 1$ to be given by a function $K$ which assigns to each object $\cY_{N+1}$ of $\cY_N$ a collection $K(\cY_{N+1})$ of families of morphisms codomain $\cY_{N+1}$, but with domains categories that are possibly different from $\cY_{N+1}$, satisfying the same conditions as those of covering families for classical Grothendieck topologies.\\

Suppose now we have a basis for a loose topology on $\cX_{n-1} = \{ \cX_N \}$ , with $\cY_n \rarr \cX_n$ a covering map in $K(\cX_n)$, and write $\cX_n = \{ \cX_{n+1} \}$ and $\cY_n =\{ \cY_{n+1} \}$, categorified. Suppose $\cX_N$ has a basis for a loose topology as well, with $\cY_{n+1} \rarr \cX_{n+1}$ a covering map in $K(\cX_{n+1})$. Then a covering map, element of a covering family in $K(\cX_{n+1},\cX_n)$, is defined to be a layered morphism:
\beq
\xymatrix{
	\cY_n \ar[r] &\cX_n \\
	\cY_{n+1} \ar@{.>}[u] \ar[r] & \cX_{n+1} \ar@{.>}[u] 
} \nonumber
\eeq
where the top map is in $K(\cX_n)$, and the bottom one is in $K(\cX_{n+1})$. Hence loose covering maps in $\cC^{[N,N+p]}$ are layered morphisms that are levelwise loose covering maps, hence also follow the same defining properties of covering maps for traditional Grothendieck topologies. Indeed, if in a diagram such as the one above, the top and bottom maps are isomorphisms, the whole diagram is itself in $K(\cX_{n+1}, \cX_n)$ by definition. It is also clear compositions are stable; if
\beq
\xymatrix{
	\cX_i \ar[r] &\cX \\
	\cX_i' \ar@{.>}[u] \ar[r] & \cX' \ar@{.>}[u] 
}  \nonumber
\eeq
is a covering map in $K(\cX', \cX)$, and if:

\beq
\xymatrix{
	\cX_{ij} \ar[r] &\cX_i \\
	\cX_{ij}' \ar@{.>}[u] \ar[r] & \cX_i' \ar@{.>}[u] 
} \nonumber
\eeq
is a covering map in $K(\cX_i', \cX_i)$, then the composition:
\beq
\xymatrix{
	\cX_{ij}' \ar[r] &\cX_i \ar[r] &\cX \\
	\cX_{ij}' \ar@{.>}[u] \ar[r] & \cX_i' \ar@{.>}[u] \ar[r] &\cX' \ar@{.>}[u]
} \nonumber
\eeq
is in $K(\cX', \cX)$. Now let:
\beq
\xymatrix{
	\cX_i \ar[r] &\cX \\
	\cX_i' \ar@{.>}[u] \ar[r] & \cX' \ar@{.>}[u] 
} \nonumber
\eeq
be elements of $K(\cX', \cX)$, indexed by i, and consider any layered morphism:
\beq
\xymatrix{
	\cZ \ar[r] &\cX \\
	\cZ' \ar@{.>}[u] \ar[r] & \cX' \ar@{.>}[u] 
} \nonumber
\eeq

we have:
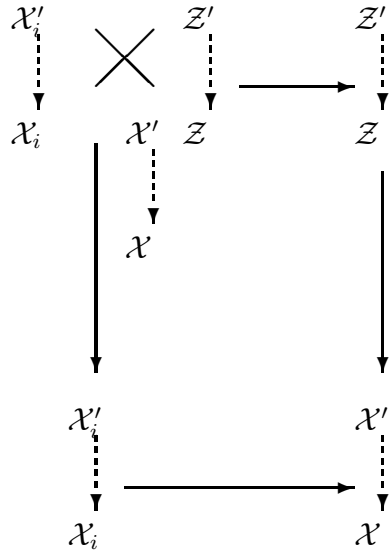
\beq
\setlength{\unitlength}{0.3in}
\begin{picture}(7,10)
	\thicklines
	\put(0,9){$\cX_i'$}
	\put(0,7){$\cX_i$}
	\multiput(0.5,8.8)(0,-0.2){5}{\line(0,1){0.1}}
	\put(0.5,7.8){\vector(0,-1){0.2}}
       \put(3,9){$\cZ'$}
	\put(3,7){$\cZ$}
	\multiput(3.5,8.8)(0,-0.2){5}{\line(0,1){0.1}}
	\put(3.5,7.8){\vector(0,-1){0.2}}
      \put(2,7){$\cX'$}
	\put(2,5){$\cX$}
	\multiput(2.5,6.8)(0,-0.2){5}{\line(0,1){0.1}}
	\put(2.5,5.8){\vector(0,-1){0.2}}
     \put(6,9){$\cZ'$}
	\put(6,7){$\cZ$}
	\multiput(6.5,8.8)(0,-0.2){5}{\line(0,1){0.1}}
	\put(6.5,7.8){\vector(0,-1){0.2}}
     \put(1,2){$\cX_i'$}
	\put(1,0){$\cX_i$}
	\multiput(1.5,1.8)(0,-0.2){5}{\line(0,1){0.1}}
	\put(1.5,0.8){\vector(0,-1){0.2}}
     \put(6,2){$\cX'$}
	\put(6,0){$\cX$}
	\multiput(6.5,1.8)(0,-0.2){5}{\line(0,1){0.1}}
	\put(6.5,0.8){\vector(0,-1){0.2}}
	\put(1.5,9){\line(1,-1){1}}
	\put(2.5,9){\line(-1,-1){1}}
	\put(1.5,7){\vector(0,-1){4}}
	\put(4,8){\vector(1,0){2}}
	\put(2,1){\vector(1,0){4}}
	\put(6.5,6.5){\vector(0,-1){3.5}}
\end{picture} \nonumber
\eeq
is isomorphic to:
\beq
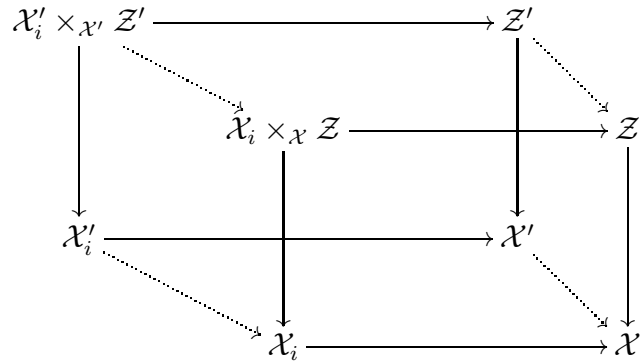
\xymatrix{
	\cX_i' \times_{\cX'} \cZ' \ar[dd] \ar@{.>}[dr] \ar[rrr] &&& \cZ' \ar[dd] \ar@{.>}[dr] \\
	& \cX_i \times_{\cX} \cZ \ar[dd] \ar[rrr] &&&\cZ \ar[dd] \\
	\cX_i' \ar@{.>}[dr] \ar[rrr] &&& \cX' \ar@{.>}[dr] \\
	& \cX_i \ar[rrr] &&& \cX
} \nonumber
\eeq
Now the top map of the back face is in $K(\cZ')$, the top map of the front face is in $K(\cZ)$, which means exactly that the following map:
\beq
\xymatrix{
\cX_i \times_{\cX} \cZ  \ar[r] &\cZ \\
\cX_i'\times_{\cX'} \cZ' \ar@{.>}[u] \ar[r] &\cZ' \ar@{.>}[u] 
} \nonumber
\eeq
is in $K(\cZ', \cZ)$.

\section{Blurry topologies}
Objects of a given type come with a notion of weak equivalence (possibly trivial). Consider the accompanying equivalence relation (generated by the relation of weak equivalence), thereby defining equivalence classes of objects of some given type. Later we will define two schemes $X$ and $Y$ to be equivalent if they have isomorphic Hochschild cohomology, thereby bypassing the need to introduce a notion of weak equivalence, and working directly with an equivalence relation. We will show if $X = \sum m_i [X_i]$, then the Hochschild cohomology of $X$ is defined by $\HHdot(X) = \otimes m_i \HHdot(X_i)$, hence if $Y = \sum n_j [Y_j]$, then $X_i \sim Y_i$ and $n_i = m_i$ for all $i$ implies $X \sim Y$. In particular, if $X = X_1 \times X_2$, $Y = Y_1 \times Y_2$, then if $X_i \sim Y_i$ for $i = 1,2$ we have $X \sim Y$. This means $Y_1 \times Y_2 \in [X_1] \times [X_2]$, implies $Y_1 \times Y_2 \sim X_1 \times X_2 \in [X_1 \times X_2]$. We are led to defining categories of type $\Gamma$ to be those for which their objects satisfy $[A] \times [B] \subset [A \times B]$.\\

Start with $\cC = \cX_0$ a category, which we suppose admits pullbacks. Objects of $\cX_0$ are of type $\Lambda_1$, $\cX_0$ itself is of type $\Lambda_0$. Objects of type $\Lambda_1$ are assumed to come with a notion of weak equivalence. Take the equivalence relation generated by it, and consider its corresponding equivalence classes. One can then write $\cX_0 = \coprod [\cX_1]$. Assume $\cX_0$ comes with a Grothendieck topology already. A basis for a \textbf{blurry topology} on $\cX_0$ is a function $K$ that assigns to each equivalence class $[\cX_1]$ a collection $K [\cX_1]$ of morphisms of $\cX_0$ with codomain $[\cX_1]$. We say $\{ [\cY] \rarr [\cX] \}$ is in $K [\cX]$ if $\{ \cY \rarr \cX \}$ is in $K(\cX)$ (or equivalently if a representative morphism is in $K(\cX)$). This defines a topology on $\cX_0 = \coprod [\cX_1]$, or a blurry topology on $\cX_0 = \{ \cX_1 \}$. Indeed, if $[\cY] \rarr [\cX]$ is an isomorphism, this means we have an isomorphism $\cY \rarr \cX$, hence $\{ \cY \rarr \cX \}$ is in $K(\cX)$, that is $\{ [\cY] \rarr [\cX] \}$ is in $K[\cX]$. Now if $\{ \phi_i: [\cX_i] \rarr [\cX] \}$ is in $K[\cX]$, if $[\cY] \rarr [\cX]$ is any morphism, we show $\{ \pi_2: [\cX_i] \times_{[\cX]} [\cY] \rarr [\cY] \}$ is in $K[\cY]$. First $\{ \underline{\phi_i}: \cX_i \rarr \cX \}$ is in $K(\cX)$, so $\{ \underline{\pi_2}: \cX_i \times_{\cX} \cY \rarr \cY \}$ is in $K(\cY)$, i.e. $\{ [ \cX_i \times_{\cX} \cY] \rarr [\cY] \}$ is in $K[\cY]$.We limit ourselves to categories $\cX_0$ of type $\Gamma : \Lambda_0$, hence $[\cX_i] \times_{[\cX]} [cY] \subset [\cX_i \times_{\cX} \cY]$. It follows $\{ [\cX_i] \times_{[\cX]} [\cY] \rarr [\cY] \}$ is in $K[\cY]$. Finally for composition if $\{ [\cX_i] \rarr [\cX] \}$ is in $K[\cX]$ and $\{ [\cX_{ij}] \rarr [\cX_i] \}$ is in $K[\cX_i]$, then $\{ \cX_i \rarr \cX \}$ is in $K(\cX)$, and $\{ \cX_{ij} \rarr \cX_i \}$ is in $K(\cX_i)$, from which it follows that $\{ \cX_{ij} \rarr \cX_i \rarr \cX \}$ is in $K(\cX)$, which means that  $\{ [\cX_{ij}] \rarr [\cX_i] \rarr [\cX] \}$ is in $K[\cX]$ since $\cX_{ij} \rarr \cX_i \rarr \cX$ is a representative morphism. Thus from a Grothendieck topology on an ordinary category of type $\Gamma$, one can derive a blurry topology. \\

Now let's see what happens if we have layered morphisms. Suppose we have a blurry topology on $\cX_0$, and $\cX_1$ and $\cY_1$ are objects of $\cX_0$, both of type $\Gamma : \Lambda_1$, categorified, with a notion of weak equivalence on their respective objects and corresponding equivalence classes, so that we can write $\cX_1 = \coprod [\cX_2]$ and $\cY_1 = \coprod [\cY_2]$. Suppose both have a loose topology defined on them. Define a \textbf{blurry loose topology} by just generalizing the notion of blurry topology on $\cX_0$: $\{ [\cY_2] \rarr [\cX_2] \}$ is in $K[\cX_2]$ if $\{ \cY_2 \rarr \cX_2 \}$ is in $K(\cX_2)$. It is not difficult to see that this also defines a loose topology on $\cX_1 = \coprod [\cX_2]$.\\

Now a diagram such as:
\beq
\xymatrix{
	[\cY_1] \ar[r] &[\cX_1] \\ 
	[\cY_2] \ar@{.>}[u] \ar[r] & [\cX_2] \ar@{.>}[u] 
}  \nonumber
\eeq

where the top horizontal map is a covering map for a blurry topology $[\tau_0]$ on $\cX_0$, with $\tau_0$ a Grothendieck topology on $\cX_0$, and the bottom map is a covering map for a blurry loose topology $[\tau_1]$ on $\cX_1$, where $\tau_1$ is a loose Grothendieck topology on $\cX_1$, together define a layered, or \textbf{powered blurry topology} $[\tau_1] \circ [\tau_0]$ on $\cX_0$. This can of course be generalized iteratively. Applying this to functors of types $F: \cC^{[N,N+p]} \rarr \cD^{[N,N+p]}$, if we have towers of equivalences in $\cC^{[N,N+p]}$, this induces level-wise quotient maps $[F] = ([F_{N+p}], \cdots, [F_N])$.\\

\section{$\bZ$-schemes-parametrized model categories}
\subsection{Parametrizations of model categories by $\bZ$-schemes}
In this section we consider full, essentially surjective functors $\xi: \PreSh(\ZSmk) \rarr \cM$ where $\cM$ is any model category. For the sake of having a good notion of space parametrizing morphisms of $\cM$, we consider functors of the form $\Sh(\ZSmk, \Nis) \rarr \cM$ instead. Morphisms of $\ZSmk$, which are elements of $\ZSmk$ themselves, map to morphisms of $\cM$. In this manner we have morphisms of $\cM$ being parametrized by $\bZ$-schemes. We now suppose we have a notion of equivalence on $\bZ$-schemes. In the next subsection we define one example of equivalence relation on such objects. Having such a notion of equivalence on $\bZ$-schemes, on which we also have a Nisnevich topology, produces a blurry Nisnevich topology $[\Nis]$.\\

Recall that a Nisnevich covering on $\ZSmk$ is a finite family of etale morphisms $\{ \fa: \Ua \rarr X \}$ in $\ZSmk$, such that for all $x \in X$, there is a $\alpha$, there is some $u \in \Ua$ with $\fa(u) = x$ and $k(u) \cong k(x)$. We generalize this notion to that of a blurry Nisnevich topology, whose coverings are given by finite families of morphisms in $\ZSmk$, $\{ [\fa]: [\Ua] \rarr [X] \}$, such that for any $x \in X \in [X]$, there is some $\alpha$, there is some $u \in \Ua \in [\Ua]$ with $\fa(u) = x$, and $k(u) \cong k(x)$, $\fa$ etale morphism in $\ZSmk$. Let $\Sh(\ZSmk, [\Nis])$ be the category of sheaves on $\ZSmk$ for that topology. If we have an object $F \in \cM[\Sh(\ZSmk, \Nis)] = \Funfes(\Sh(\ZSmk, \Nis), \cM)$, we have an induced morphism $[F]: \Sh(\ZSmk, [\Nis]) \rarr \cM/\!\!\sim$. This is the first stage with $\tau_0 = \Nis$. Now one could stop there or use powered topologies. If that is the case, we categorify each object $X = \sum m_i [X_i]$ of $\ZSmk$, defining it as a category with objects $[X_i]$, with morphisms those morphisms in $\Smk$. $\bZ$-schemes $X$ are of type $\Lambda_1$. Now covering families on $X$ are finite families of etale morphisms $\{ f_{ij}: Y_j \rarr X_i \}$ such that for any $x_i \in X_i$, there is some $y_j \in Y_j$ such that $f_{ij}(y_j) = x_i$ and $k(y_j) \cong k(x_i)$, where those $Y_j's$ originate from some $Y = \sum n_j[Y_j]$, with a preexisting morphism $Y \rarr X$ in $\ZSmk$. Here in the covering family we have not used brackets for objects of $\SmCor(k)$ to avoid confusion with equivalence classes of schemes. This gives us a loose topology $\tau_1$ on $X$. We can consider the associated blurry loose topology $[\tau_1]$, resulting in a powered blurry topology $[\tau_1] \circ [\tau_0]$ on $\ZSmk$. Moving forward we can further categorify each $X_i$, smooth over $k$, so of finite type, hence it can be covered by finitely many affine schemes $(\Spec R_{ik}, \cO_{ik})$, objects, with morphisms morphisms of affine schemes. If we have a morphism from $Y_j$ to $X_i$ we do likewise for $Y_j$, covered by $(\Spec S_{jl}, \cO_{jl})$, and coverings are finite families of etale morphisms of ringed spaces $\Spec S_{jl} \rarr \Spec R_{ik}$, giving us a loose topology $\tau_2$, with an associated blurry loose topology $[\tau_2]$, yielding a powered blurry topology $[\tau_2] \circ [\tau_1] \circ [\tau_0]$ on $\ZSmk$. We can pursue in this manner as many times as needed, provided subsequent topologies can be defined. This gives rise to $\Sh(\ZSmk, [\tau_2] \circ [\tau_1] \circ [\tau_0])$. If we have a functor $F: \Sh(\ZSmk, \tau_2 \circ \tau_1 \circ \tau_0) \rarr \cM$, this induces $[F] = ([F_2], [F_1], [F_0]): \Sh(\ZSmk, [\tau_2] \circ [\tau_1] \circ [\tau_0]) \rarr \cM/\!\!\sim$. Another alternative consists in not having a notion of equivalence on $\ZSmk$, but to have an interval object $I$ instead on the site $(\ZSmk, \text{Nis})$, such as the affine line $\bA^1$, and this is the point of view adopted in \cite{Ka}. What is studied in \cite{Ka} is the topos $\cM^{(\Smk, \text{Nis}, \bA^1)}$, for $\cM$ a left proper, combinatorial simplicial model category. In the present paper we put no restriction on our model categories $\cM$ for the simple reason that we do not take a Bousfield localization of our topos $\Fun(\Sh(\ZSmk, \Nis), \cM)$. Nevertheless we will come back later to a generalization of the work done in \cite{Ka} to contrast this with using equivalences on schemes.\\

\subsection{Equivalence relations on $\ZSmk$}
For our notion of equivalence, we will use Hochschild cohomology on $\ZSmk$, which we will define as a generalization of the usual Hochschild cohomology of schemes as developed in \cite{GS} and \cite{S} in particular, but where some relevant treatments can also be found in \cite{K}, \cite{Ku}. The idea of using Hochschild cohomology is derived from the fact that since one has functors from $\PreSh(\ZSmk)$ to $\cM$, one would want equivalent $\bZ$-schemes to map to the same object. If we regard functors as representations, one would think in algebraic terms about Morita equivalent algebras, which is trivial for commutative rings. From Morita theory one can easily think of Hochschild cohomology. The latter is not trivial on $\Smk$ however. Recall, from \cite{GS} and \cite{S}, that for $X$ a separated scheme of finite type over $k$, $\cF$ a sheaf of $\cO_X$-modules, one can define the Hochschild cohomology of $X$ with coefficients in $\cF$ by:
\beq
\xH^n(\cO_X, \cF) = \Ext^n_{\OXX}(\cO_X, \cF) \nonumber
\eeq
where $\cF$ is regarded as a sheaf of $\OXX$-modules via the diagonal functor. Define the Hochschild cohomology of a scheme $X$ as $\xH^n(X) = \xH^n(\cO_X, \cO_X)$, and we define two schemes $X$ and $Y$ to be equivalent if $\Hdot(X) \cong \Hdot(Y)$. We now generalize this $\bZ$-schemes. Let $X = \sum m_i [X_i]$ and $Y = \sum n_j[Y_j]$. To be explicit, in $\Smk$ we have $\Hdot(X) = \Extdot_{\OXX}(\Delta_* \cO_X, \Delta_* \cO_X)$. We have $X \times  X = \sum m_i X_i \times X_i$, so that $\OXX = \otimes \cO_{m_i X_i \times X_i} = \otimes m_i \cO_{X_i \times X_i}$, and it also follows $\Delta_*( \cO_X) = \Delta_*( \otimes m_i \cO_{X_i}) = \otimes m_i \Delta_*( \cO_{X_i})$. Now:
\begin{align}
	\Hdot(X) &= \Extdot_{\OXX}( \Delta_* \cO_X, \Delta_* \cO_X) \nonumber \\
	&= \Extdot_{ \otimes m_i \cO_{X_i \times X_i}} ( \otimes m_i \Delta_* \cO_{X_i}, \otimes m_i \Delta_* \cO_{X_i}) \nonumber \\
	&= \otimes m_i \Extdot_{\cO_{X_i \times X_i}} ( \Delta_* \cO_{X_i}, \Delta_* \cO_{X_i}) = \otimes m_i \Hdot(X_i) \nonumber
\end{align}
thus we can define two objects $ X = \sum_{i \in I} m_i [X_i]$ and $Y = \sum_{j \in J} n_j [Y_j]$ of $\ZSmk$ to be equivalent if the indexing sets $I = J$, $m_i = n_i$ for all $i \in I$, and $\Hdot (X_i) \cong \Hdot (Y_i)$ for all $i \in I$. This then defines a notion of Hochschild equivalence on $\ZSmk$.\\

Another definition of Hochschild cohomology of $\bZ$-schemes we can use is the Grothendieck-Loday definition of such, as presented in \cite{S} for $\Smk$. Recall that if $A$ is an algebra over a field $k$, letting $A^e = A \otimes_k A$, we can define the bar complex by $\Bdot(A) = A \otimes_k A^{\otimes \ctrdot} \otimes_k A$. Then we define $\Cdot(A) = A \otimes_{A^e} \Bdot(A) = A \otimes_k A^{\otimes \ctrdot}$. If $X$ is a smooth scheme over $k$, we define a presheaf on $X$ by letting $\cCdot(U) = \Cdot(\Gamma(U, \cO_X))$. We denote by $a \cCdot$ the associated sheaf, where $a$ is the sheafification functor. It is a sheaf of $\cO_X$-modules. Now if $\cFdot$ is a chain complex of sheaves of $\cO_X$-modules, if $\cG$ is a $\cO_X$-module with an injective resolution $0 \rarr \cG \rarr \cIdot$, then we define the hyperext by:
\beq
\bExt^n_{\cO_X}( \cFdot, \cG) = \xH^n( \Hom_{\cO_X}(\cFdot, \cIdot)) \nonumber 
\eeq
In $\ZSmk$, $\cIdot(\otimes m_i \cG_i) = \otimes m_i \cIdot \cG_i$.
With this in hand we can define the Grothendieck-Loday type definition of Hochschild cohomology of schemes with values in $\cF$ by $\HHdot(X, \cF) = \bExtdot_{\cO_X} (a \cCdot, \cF)$, and the Hochschild cohomology of schemes by $\HHdot(X) = \HHdot(X, \cO_X) = \bExtdot_{\cO_X}( a \cCdot, \cO_X)$. As usual, we say $X \sim Y$ in $\Smk$ if and only if $ \HHdot(X) \cong \HHdot(Y)$. We now generalize this definition to $\ZSmk$. First $\Gamma(-,\cO_X) = \otimes \Gamma_i (-, \cO_{m_iX_i}) = \otimes m_i \Gamma_i(-, \cO_{X_i})$, where $\Gamma_i$ is the section functor on $X_i$. Since $\cCdot(U) = \Cdot(\Gamma(U, \cO_X))$ defines a presheaf on $X$, it follows $\cC_{\ctrdot, mX} = \Cdot(\Gamma(U, \cO_{mX})) = \Cdot(m \Gamma(U, \cO_X))$ is also equal to $m \cC_{\ctrdot, X}$, from which it follows that if $X = \sum m_i [X_i] \in \ZSmk$:
\begin{align}
	\HHdot(X) &= \Hdot(\Hom_{\otimes m_i \cOXi}( a C_{\ctrdot}(\otimes m_i \Gamma_i(-, \cOXi)), \otimes m_i \cIdot \cOXi)) \nonumber \\
	&=\Hdot(\Hom_{\otimes m_i \cOXi}(a (\otimes m_i C_{\ctrdot} \Gamma_i(-, \cOXi)), \otimes m_i \cIdot \cOXi)) \nonumber \\
	&=\Hdot(\Hom_{\otimes m_i \cOXi}( \otimes m_i a  C_{\ctrdot} \Gamma_i(-, \cOXi), \otimes m_i \cIdot \cOXi)) \nonumber \\
	&= \Hdot( \otimes \Hom_{m_i \cOXi}(m_i a C_{\ctrdot} \Gamma_i(-,\cOXi), m_i \cIdot \cOXi)) \nonumber \\
	&= \Hdot( \otimes m_i \Hom_{\cOXi}(a C_{\ctrdot} \Gamma_i(-,\cOXi), \cIdot \cOXi)) \nonumber \\
	&= \otimes m_i \Hdot(\Hom_{\cOXi}(a C_{\ctrdot} \Gamma_i(-,\cOXi), \cIdot \cOXi)) \nonumber \\
	&= \otimes m_i \HHdot(X_i) \nonumber
\end{align}
hence we define, again, $X = \sum_{i \in I} m_i [X_i]$ and $Y = \sum_{j \in J} n_j [Y_j]$ in $\ZSmk$ to be equivalent if $I = J$, $m_i = n_i$ for all $i \in I$, and $\HHdot(X_i) \cong \HHdot(Y_i)$ for all $i \in I$.

\subsection{$\bA^1$-homotopy category of $\bZ$-schemes as parameter space}
As pointed out above, an alternative to using a notion of equivalence on $\bZ$-schemes consists in using an interval object $I$ on $(\ZSmk, \Nis)$. Naturally one would take $I = \bA^1$, as done for the homotopy theory of schemes (\cite{VM}), developed from presheaves of simplicial sets. We will use a variant of such a construction, not regarding $\bA^1$ as an interval object, but just as a presheaf, we will localize with respect to $\bA^1$-local maps, and then use the Nisnevich topology on such a localization. The construction is fairly transparent.\\

Recall that in this work, we consider presheaves of sets. In this section in particular, we consider objects of $\PreSh(\ZSmk, \Nis, \bA^1)$. In $\bA^1$-homotopy theory of schemes however, we work with simplicial sheaves. Thus we regard presheaves of sets as constant simplicial presheaves. Consider $\Hom_{\ZSmk}(-,\bA^1)$, the representable presheaf associated with $\bA^1$, that we still denote by $\bA^1$. Consider the functor category $\Fun((\ZSmk)^{\op}, \SetD)$. From there we essentially follow \cite{Hi}. Recall that if $\cM$ is a model category, $S$ is a class of maps in $\cM$, we can define a model category structure on the underlying category of $\cM$, denoted $L_S \cM$, for which weak equivalences are $S$-local equivalences in $\cM$, cofibrations are those of $\cM$, and fibrations are those maps that have the right lifting property with respect to cofibrations that are also $S$-local equivalences. Recall also what those are: an object $W$ of $\cM$ is said to be $S$-local if it is fibrant, and if for any $f:A \rarr B$ in $S$, the induced map of homotopy function complexes $f^*: \map(B,W) \rarr \map(A,W)$ is a weak equivalence of simplicial sets. A map $g:X \rarr Y$ in $\cM$ is a $S$-local equivalence if for any $S$-local object $W$, the induced map of homotopy function complexes $g^*: \map(Y,W) \rarr \map(X,W)$ is a weak equivalence of simplicial sets. It is a fact that if $\cM$ is a left proper, cellular model category, $S$ a set of maps in $\cM$, then the left Bousfield localization $L_S \cM$ of $\cM$ exists. Now $\SetD$ is a left proper cellular model category, $\ZSmk$ is a small category, hence $\SetD^{\ZSmk^{\op}}$ is also left proper cellular (\cite{Hi}). Denote it by $\cM$, let $S$ be the set of projection maps $\{ F \times  \bA^1 \rarr F \}$ for $F \in \cM$. An object $G$ of $\cM$ is $S$-local, or $\bA^1$-local, if it is fibrant, and for any $p: F \times \bA^1 \rarr F$ in $S$, the induced map of homotopy function complexes $p^*: \map(F,G) \rarr \map(F \times \bA^1, G)$ is a weak equivalence in $\SetD$. Then $\alpha: F \Rightarrow H$ is an $\bA^1$-local equivalence if for any $\bA^1$-local object $G$, the induced map $\alpha^*: \map(H,G) \rarr \map(F,G)$ is a weak equivalence in $\SetD$. We consider $L_{\bA^1} \SetD^{(\ZSmk)^{\op}}$, the left Bousfied localization of $\SetD^{\ZSmk^{\op}}$ with respect to $S = \{ F \times \bA^1 \rarr F \}$, then use the Nisnevich topology on $\ZSmk$, giving rise to $\Xi = L_{\bA^1} \SetD^{(\ZSmk^{\op}, \Nis)}$, and hence to $\Lambda = \Sh_{\bA^1}(\ZSmk, \Nis) \subset \Xi$, and finally functors $\Lambda \rarr \cM$ for $\cM$ any model category can be regarded as providing parametrizations of $\cM$ by $\bA^1$-homotopic $\bZ$-schemes, where now we consider $\bA^1$-local objects instead of Hochschild equivalent $\bZ$-schemes.

\end{document}